\documentclass[11pt,reqno]{amsart}
\usepackage{xifthen}
\usepackage{enumerate}
\usepackage{bbm}
\usepackage{comment}
\usepackage{tikz}
\usepackage{hyperref}
\usepackage[cp1250]{inputenc}
\DeclareInputText{163}{\L}
  
\usepackage{fancyhdr}
\newcommand{\nN}{{\mathbf{N}}}

\newcommand{\kK}{{\mathbf{K}}}

\numberwithin{equation}{section}
\newtheorem{thm}{Theorem}[section]
\newtheorem{cor}{Corollary}[section]
\newtheorem{lem}{Lemma}[section]
\newtheorem{prop}{Proposition}[section]

\newtheorem{assume}{Assumption}[section]
\theoremstyle{definition}

\newtheorem{rmk}[thm]{Remark}

\newcommand{\E}{\mathbb{E}}
\newcommand{\N}{\mathbb{N}}
\newcommand{\V}{\mathbb{V}}
\newcommand{\vep}{\varepsilon}
\newcommand{\K}{\mathbb{K}}

	\renewcommand{\P}{\mathbb{P}}

\makeatletter
\@namedef{subjclassname@2020}{%
  \textup{2020} Mathematics Subject Classification}
\makeatother

\begin{document}

\title{k-core in percolated dense graph sequences}
\author[]{Erhan Bayraktar} \address{Department of Mathematics, University of Michigan}
\email{erhan@umich.edu}
\author[]{Suman Chakraborty}
\address{Department of Mathematics and Computer Science, Eindhoven University of Technology}
\email{s.chakraborty1@tue.nl}
\author[]{Xin Zhang} 
\address{Department of Mathematics, University of Michigan}
\email{zxmars@umich.edu}

\date{\today}

\maketitle

\begin{abstract}
We determine the size of $k$-core in a large class of dense graph sequences. Let $G_n$ be a sequence of undirected, $n$-vertex graphs with edge weights $\{a^n_{i,j}\}_{i,j \in [n]}$ that converges to a kernel $W:[0,1]^2\to [0,+\infty)$ in the cut metric. Keeping an edge $(i,j)$ of $G_n$ with probability $\min \{ {a^n_{i,j}}/{n},1 \}$ independently, we obtain a sequence of random graphs  $G_n(\frac{1}{n})$. Denote by $\mathcal{A}$ the property of a branching process that the initial particle has at least $k$ children, each of which has at least $k-1$ children, each of which has at least $k-1$ children, and so on. Using branching process and the theory of dense graph limits, under mild assumptions we obtain the size of $k$-core of random graphs $G_n(\frac{1}{n})$,
\begin{align*}
\text{size of $k$-core of } G_n\left(\frac{1}{n}\right) =n \mathbb{P}_{X^W}\left(\mathcal{A}\right) +o_p(n). 
\end{align*}
Our result can also be used to obtain the threshold of appearance of a $k$-core of order $n$. 
\end{abstract}

\section{Introduction}

For an integer $k\geq 2$, the $k$-core of a graph $G$ is the largest
induced subgraph of $G$ with the minimum degree at least $k$. It was first introduced by Bollob\'{a}s in \cite{MR777163} to find large $k$-connected subgraphs, and since then several studies have been devoted to investigating the existence and the size of $k$-core. Apart from theoretical interests, $k$-core has been applied to the study of social networks \cite{MR3379018, 10.1145/2588555.2610495}, graph visualizing \cite{10.5555/2976248.2976254, Carmi11150}, biology \cite{You:2013aa}. See also \cite{kong2019k} for an extensive discussion on its applications.   In the seminal paper \cite{pittel1996sudden}, Pittel, Spencer and Wormald determined the threshold for the appearance of a non-empty $k$-core in Bernoulli random graphs and uniform random graphs. The size of $k$-core have been studied in different random graph ensembles such as Bernoulli random graphs \cite{MR1120887},
uniformly chosen random graphs and hypergraphs with specified degree sequence \cite{cooper2004cores, fernholz2004cores,janson2007simple, janson2008asymptotic,  molloy2005cores}, Poisson cloning model \cite{kim2007poisson}, and the pairing-allocation model \cite{cain2006encores}. While almost all the previous work focused on $k$-core of homogeneous random graphs, Riordan \cite{MR2376426} determined the asymptotic size of $k$-core for a large class of inhomogeneous random graphs, which was introduced in \cite{MR2337396}.

In this article we study the asymptotic size of $k$-core in random subgraphs of convergent dense graph sequences. Let $G_n$ be a sequence of undirected weighted graphs on $n$ vertices with edge weights $\{a^n_{i,j}\}$ that converges to a graphon $W$. For some $c>0$, we keep an edge $(i,j)$ of $G_n$ with probability $\min \{ {ca^n_{i,j}}/{n},1 \}$ independently, and denote the resulting random graph by $G_n(\frac{c}{n})$. For any kernel $W$, we can associate it with a branching process $X^W$, i.e., the number of children of a particle with type $x$ has Poisson distribution with parameter $\int W(x,y) dy$ (see Section~\ref{sec2} for precise definition). Under some mild conditions, we show that 
\begin{align}\label{eq:eq}
\text{size of $k$-core of } G_n\left(\frac{c}{n}\right)= n\P_{X^{cW}}(\mathcal{A}) + o_p(n),
\end{align}
where $\mathcal{A}$ is the event that the initial particle has at least $k$ children, each of which has at least $k-1$ children, each of which has at least $k-1$ children, and so on. 

Our contribution is two-fold.  First, recall from \cite{lovasz2006limits} that every dense graph sequence has a convergent subsequence, and therefore our result applies to a very large class of dense graph sequences. In particular, our result together with \cite[Lemma 1.6]{bollobas2010cut} recovers the main result in  \cite{MR2376426} for bounded graphons. An important application of our result is quasi-random graphs (see e.g. \cite{chung1989quasi}), roughly, they are dense graph sequences that converge to a constant (non-zero) limit, such as Paley graphs (see  \cite{janson2011quasi}). It is worth pointing out that the construction of various quasi-random graphs involves very different techniques, ranging from finite geometries to linear algebra (see \cite{krivelevich2006pseudo}). This article provides a unified way to determine the size of the $k$-core in random subgraphs of such graphs up to the first-order term. Moreover, various quantities such as the size of the largest connected component in random subgraphs of quasi-random graphs have been studied recently \cite{nachmias2009mean, krivelevich2016phase, diskin2021site, chakraborty2017site}. As far as we know, other than the present work no result is known about the size of the $k$-core in random subgraphs of quasi-random graphs. Our result provides much more. There are plenty of examples of dense random graph models, which are not quasi-random that are known to converge to a bounded graphon (see \cite{basak2016large,bhamidi2018weighted,chatterjee2013estimating, chatterjee2011random}), where our result could be applicable. Second, as a byproduct of our proof of the main result, for any sequence of kernels $W_n$ satisfying some mild assumptions that converges to $W$ we have that 
\begin{align*}
\mathbb{P}_{X^{W_n}}(\mathcal{A}) \rightarrow \mathbb{P}_{X^{W}}(\mathcal{A}),
\end{align*}
a new continuity result concerning branching processes, which we believe is of independent interest. Even though the theory of graph limits received enormous attention in the last two decades, the only result alike that we can find is \cite[Theorem 1.9]{bollobas2010cut}, which concerns with the survival probability of a branching process.  \par
 Let us now point out the main difficulties and the ideas of our proof. Firstly, the proof technique in \cite{MR2376426} does not readily apply in our case.  For example, in this article, the proof of the upper bound of the size of $k$-core is done by upper bounding the size of the $k$-core in terms of the probability of the event $\mathcal{A}$. The approximation of the probability of the event $\mathcal{A}$ is subtle. We first approximate the probability of $\mathcal{A}$ in terms of homomorphism densities of appropriate subgraphs and then estimate this probability using the fact that homomorphism densities are continuous in cut metric; see e.g. \cite{MR2599196,lovasz2006limits}. The proof of the lower bound is more delicate. First, we approximate $W$ by a sequence of finitary kernels $F_m$ as in \cite{MR2337396}. Then, we show that for each fixed $m$, the branching process $X^{n}$ associated with $G_n$ contains $X^{(1-\vep_m)F_m}$ as a subset for some $\vep_m$ with $0<\vep_m<\frac{1}{m}$ when $n$ is large enough. To conclude the lower bound, we prove a continuity property, which is non-trivial and relies on the properties of the cut norm, and finally, we invoke a result (minor variant) from \cite{MR2376426}. 

The rest of the paper is organized as follows. In Section~\ref{sec2}, we present our main results with some discussions. In Section~\ref{sec:422pm10dec20} and Section~\ref{sec:446pm10dec20}, we prove the upper bound and lower bound of the size of $k$-core respectively.

\section{Main results and discussions}\label{sec2}

We now recall a few definitions to state our results. A graphon (or kernel) is defined to be a symmetric measurable function $W: I \times I \to [0,\infty)$, where  $I := [0,1]$. Take $\mathcal{W}$ to be the space of graphons. The cut norm of $W: I\times I \rightarrow \mathbb{R}$ (signed graphon) is defined by
$$\|W\|_\square := \sup_{S,T \in \mathcal{B}(I)} \left| \int_{S \times T} W(u,v)\,du\,dv \right|,$$
and the cut metric between two graphons $W_1$ and $W_2$ is defined by 
$$d_\square(W_1,W_2) := \|W_1-W_2\|_\square.$$
 An undirected finite graph $G_n$ with adjacency matrix $(a^n_{i,j})_{i,j=1}^n$ can be embedded into a symmetric kernel in a natural way
\begin{equation}\label{eqn:satapr11348pm}
	W_{G_n}(x,y) = \sum_{i,j=1}^n a^n_{i,j} {\mathbbm{1}}_{J_i^n}(x){\mathbbm{1}}_{J_j^n}(y),
\end{equation}
where $J_1^n = [0,\frac{1}{n}]$ and for $i= 2,3,\ldots,n$,  $J_i^n = \left(\frac{i-1}{n},\frac{i}{n}\right]$.

Let $G_n$ be a sequence of simple graphs on $n$ vertices with edge weights $\{a^n_{i,j}\}$ that converges to a kernel $W$. For some $c>0$, we keep an edge $(i,j)$ of $G_n$ with probability $\min \{ {ca^n_{i,j}}/{n},1 \}$ independently, and denote the resulting random graph by $G_n(\frac{c}{n})$. Here and throughout the paper we assume that edge weights $a^n_{i,j}$ are uniformly bounded by $\bar{a}>0$, and therefore for sufficiently large $n$ we will have $ \min \{ {ca^n_{i,j}}/{n},1 \} = {ca^n_{i,j}}/{n}$. Since retaining every edge independently is nothing but the bond-percolation on the graph , we call $G_n\left(\frac{c}{n}\right)$ percolated graph sequence (bond-percolation on arbitrary dense graph sequences was first studied in \cite{MR2599196}). Note that if we percolate on a dense graph sequence, where number of edges is of order $n^2$, the resulting graph becomes sparse, that is, it has order of $n$ many edges. Our aim is to study the size of the $k$-core of the random graph sequence $G_n\left(\frac{c}{n}\right)$. 

We will heavily use the branching process $X^W$  associated with the kernel $W$. The process starts with a single particle with type $x_0$, which is chosen uniformly from $[0,1]$. Conditional on generation $t$,  each member in generation $t$ has offsprings in the next generation independent of each other, and everything else. The number of children with types in a set $A \subset [0,1]$ is Poisson with parameter $\int_{A} W(x,y) \,dy$, and these numbers are independent for disjoint sets.\par

Let $\mathcal{A}_d$  be the event that the root has at least $k$ children, each of these $k$ children has at least $k-1$ children, each of those second generation of children has another $k-1$ children and so on until the $d$-th generation. Define $\mathcal{A}= \cap_{d=1}^{\infty} \mathcal{A}_d$. Let $C_k(G)$ denote the size of the $k$-core of a graph $G$. We are now ready to discuss our main result, which provides asymptotic size of the $k$-core in percolated dense graph sequences. As usual, if $A_n$ is a sequence of random variables, we write $A_n=o_p(f(n))$ if $A_n/f(n) \to 0$ in probability; if $A_n$ is a sequence of real numbers, we write $A_n=o(f(n))$ if $A_n/f(n) \to 0$. First let us make the following assumption.

\begin{assume}\label{assume2} $\lambda \to \P_{X^{\lambda W}}(\mathcal{A})$ is continuous at $\lambda =c$ with $c>0$.
\end{assume}

\begin{thm}\label{thm:554pm07dec20}
	Let $G_n$ be a  sequence of graphs with non-negative edge weights which are bounded above by a constant $\bar{a}>0$. Suppose that $G_n$ converges to a graphon $W$ as $n \rightarrow \infty$ and that the Assumption~\ref{assume2} holds. Then we have that 
\begin{equation}\label{eqn:thuapr9333pm}
		C_k\left(G_n\left(\frac{c}{n}\right)\right) =n \mathbb{P}_{X^{cW}}\left(\mathcal{A}\right) +o_p(n). 
\end{equation}
\end{thm}

It suffices to prove the case $c=1$ in Theorem~\ref{thm:554pm07dec20}. To see this, let $G_n$ be a graph with edge weights $\{a^n_{i,j}\}$ and consider another graph $G_n'$ with edge weights $\{ca^n_{i,j} \}$. Therefore the random subgraphs $G_n\left(\frac{c}{n}\right)$ and $G_n'\left(\frac{1}{n}\right)$ are equal in distribution. Finally by our assumption $G_n$ converges to $W$ and this gives $G_n'$ converges to $cW$. The result  \eqref{eqn:thuapr9333pm} then follows from the result with $c=1$. 

Our proof of \eqref{eqn:thuapr9333pm} is divided into two parts, which will be given in the next two sections. We should remark that for the proof of upper bound, we only need the assumption that the edge weights of $G_n$ are uniformly bounded above by $\bar{a}$ and $G_n \to W$. Assumption~\ref{assume2} is used only in the proof of lower bound in Section~\ref{sec:446pm10dec20}.

\begin{rmk}\label{rmk:351pm16dec20}
In Theorem~\ref{thm:554pm07dec20}, note that $\mathbb{P}_{X^{cW}}\left(\mathcal{A}\right)$ could be zero and in that case we will only be able to say that there is no `giant' $k$-core (as usual by `giant' we mean `of size order of $n$'). From Theorem~\ref{thm:554pm07dec20} one can also obtain the emergence threshold for the giant $k$-core from the function $c \rightarrow \mathbb{P}_{X^{cW}}\left(\mathcal{A}\right)$. More precisely, if there is a point $c_0>0$ such that for $0\leq c< c_0$, $\mathbb{P}_{X^{cW}}\left(\mathcal{A}\right)=0$ and for $c>c_0$, $\mathbb{P}_{X^{cW}}\left(\mathcal{A}\right)>0$, then $c_0$ will be the threshold for the appearance of giant $k$-core. The other discontinuity points could be studied from this function as well.
\end{rmk}

\begin{rmk}
It is not difficult to show that the function $\lambda\to\mathbb{P}_{X^{\lambda W}}\left(\mathcal{A}\right)$ is continuous from above (see Section 3.2 in \cite{MR2376426}), therefore we only assume that the function is continuous from below. We now discuss why it is not possible to get rid of the Assumption \ref{assume2} in Theorem~\ref{thm:554pm07dec20}. Firstly, let $c_k$ be the threshold of the appearance of a $k$-core in the binomial random graph on $n$ vertices with edge probability $\tfrac{c}{n}$. Then from \cite[Theorem 1.3 (ii)]{janson2008asymptotic} we get that the assertion in Theorem~\ref{thm:554pm07dec20} does not hold at the threshold (discontinuity point), which tells that the Assumption \ref{assume2} is optimal. More precisely, at the threshold the $k$-core is empty with probability bounded away from $0$ and $1$ for large $n$.  
 \par
 Secondly, there could be more than one discontinuity points, and they could appear in different (non-trivial) ways. Let us now roughly explain how such discontinuities could appear and refer the interested reader to the end of Section 3.1 in \cite{MR2376426} for the details. Consider the following graphon 
\begin{align*}
 W(x,y)=
\begin{cases} 
 2000 \;\; &(x,y) \in [0,1/2]^2 \\
 2  \;\;\;\;\;\;\;\; &(x,y) \in (1/2,1]^2\\
  \frac{1}{100} \;\; &\text{otherwise}.
\end{cases}
\end{align*}

It is not difficult to show that the $3$-core first appears in the subgraph induced by the vertices that correspond to the bottom left part ($[0,1/2]^2$) of the graphon, and this emerges near $\tfrac{c_3}{1000}$, where $c=c_3$ is the threshold of appearance of $3$-core in the binomial random graph on $n$ vertices with edge probability $\tfrac{c}{n}$. It could also be shown that $\P_{X^{\lambda W}}(\mathcal{A})$ has another discontinuity near $\lambda=c_3$, that is when the subgraph induced by the vertices that correspond to the top right part ($(1/2,1]^2$) of the graphon will have a $3$-core. There is another more straightforward way in which discontinuity could appear in $\P_{X^{\lambda W}}(\mathcal{A})$. Consider the graphon   
\begin{align*}
 W(x,y)=
\begin{cases} 
2 \;\; &(x,y) \in [0,1/2]^2 \\
 1  \;\;\;\;\;\;\;\; &(x,y) \in (1/2,1]^2\\
0 \;\; &\text{otherwise}.
\end{cases}
\end{align*}
It is easy to see that the emergence threshold of the $k$-core in the subgraph induced by the vertices that correspond to the bottom left ($[0,1/2]^2$) part of the graphon, and the subgraph induced by the vertices that correspond to the top right part ($(1/2,1]^2$) of the graphon will be different.
  
\end{rmk}

As a byproduct in the proof of Theorem~\ref{thm:554pm07dec20}, we also obtain a result regarding branching processes that might be of independent interest.   
\begin{prop}\label{thm:546pm09dec20}
Let $W_n$ be a sequence of graphons such that $d_\square(W_n,W) \rightarrow 0$. Also suppose $\lambda \to \P_{X^{\lambda W}}(\mathcal{A})$ is continuous from below at $\lambda =1$. Then it holds that 
\begin{equation}\label{eqn:545pm09dec20}
	\mathbb{P}_{X^{W_n}}(\mathcal{A}) \rightarrow \mathbb{P}_{X^{W}}(\mathcal{A}),
\end{equation}
as $n\rightarrow \infty$.
\end{prop}
\begin{proof}
It is proved in Propositions ~\ref{prop:branchingeq},~\ref{prop:continuity}.
\end{proof}

Let us point out that Proposition~\ref{thm:546pm09dec20}  has the following important consequence. Note that the function $\lambda \to \P_{X^{\lambda W}}(\mathcal{A})$ is non-decreasing, and therefore it can have at most countably many discontinuity points. Hence for almost every positive $c$, the next corollary provides a way to approximate the size of $k$-core using only $G_n$.

\begin{cor}\label{cor:736pm09dec20}
 	Let $G_n$ be a  sequence of graphs with non-negative edge weights which are bounded above by a constant $\bar{a}>0$. Suppose that $G_n$ converges to a graphon $W$ as $n \rightarrow \infty$, where $\lambda \to \P_{X^{\lambda W}}(\mathcal{A})$ is continuous at $\lambda =c$, then
 
 \begin{equation}\label{eqn:912pm09dec20}
 	C_k\left(G_n\left(\frac{c}{n}\right)\right) =n \mathbb{P}_{X^{cW_{G_n}}}\left(\mathcal{A}\right) +o_p(n). 
 \end{equation}
\end{cor}

\begin{proof}[Proof of Corollary \ref{cor:736pm09dec20}]
The proof is immediate using Theorem~\ref{thm:554pm07dec20} and Proposition~\ref{thm:546pm09dec20}.
\end{proof}

\section{Proof of the upper bound in Theorem~\ref{thm:554pm07dec20}}\label{sec:422pm10dec20}
We will first prove the upper bound, i.e., 
\begin{align*}
C_k\left(G_n\left(\frac{1}{n}\right)\right)  \leq n \mathbb{P}_{X^W}\left(\mathcal{A}\right) +o_p(n). 
\end{align*} 
The idea is as follows: if a vertex $v$ of a graph is in the $k$-core, then for any $d>0$ either $v$ has property $\mathcal{A}_d$ or $v$ is contained in a cycle of length smaller than $2d$. Since the probability of occurrence of  short cycles is small for large enough $n$, the probability that $v$ is in the $k$-core is bounded above by the probability of having property $\mathcal{A}_d$. Therefore to prove the upper bound, we explicitly calculate the probability of event $\mathcal{A}_d$ using homomorphism density, and a tightness argument. Finally, by letting $d \to \infty$, we obtain that $C_k\left(G_n\left(\frac{1}{n}\right)\right)  \leq n \mathbb{P}_{X^W}\left(\mathcal{A}\right) +o_p(n)$. Note that we do not need the  limit $W$ to be bounded below by a constant or the continuity assumption for the upper bound.

Let us construct a branching process $^{*}X^{n}$ associated with the random graph $G_n\left(\frac{1}{n}\right)$.  $^*X^n$ has $n$-types of offsprings $1,2,\ldots,n$. It starts with a single particle whose type is chosen uniformly from $1, 2,\ldots, n$. Conditioning on generation $t$, each member of generation $t$ has offsprings in the next generation independent of each other, and everything else. The number of $j$-offspring of a particle of type $i$ is Bernoulli$(a^n_{i,j}/n)$. 

We will also use another branching process where number of $j$-offsprings of a particle of type $i$ is Poisson$(a^n_{i,j}\rho_n)$, where $\rho_n \geq \frac{1}{n}$ is to be determined. We denote this process by $X^{n,\rho_n}$ (simply by $X^n$ if $\rho_n=\frac{1}{n}$). By taking $\rho_n=\frac{1}{n-\bar{a}}$, the Poisson branching process $X^{n, \rho_n}$ stochastically dominates, in the first order, $^* X^n$ for $n>3\bar{a}$. To see this, it is sufficient to show the following inequality for any $i,j \in [n]$ 
$$
\mathbb{P}\left(\text{Poisson}\left( a_{i,j}^n \rho_n\right) >t\right) \geq \mathbb{P}\left(\text{Bernoulli}\left(\frac{a_{i,j}^n}{n}\right)>t\right)
$$
It is trivial for $t\geq 1$ and $t<0$. We need to check only for $t=0$. It can be easily verified that the above inequality is equivalent to 
$$
\frac{n \rho_n \left(1-e^{-a_{i,j}^n \rho_n }\right)}{a_{i,j}^n \rho_n}  \geq 1.
$$
For $n > 3 \bar{a}$, we have that $a_{i,j}^n \rho_n = \frac{a_{i,j}^n}{n-\bar{a}} < 1/2$, and hence according to the Taylor expansion of $e^{-a_{i,j}^n \rho_n}$,
\begin{align*}
\frac{n \rho_n \left(1-e^{-a_{i,j}^n \rho_n }\right)}{a_{i,j}^n \rho_n} > n \rho_n (1- a_{i,j}^n \rho_n /2 ) \geq (1+\bar{a}\rho_n )(1-\bar{a}\rho_n /2) \geq 1. 
\end{align*}

Note that  we can write
$$
C_k\left(G_n\left(\frac{1}{n}\right)\right) = \sum_{v \in [n]} {\mathbbm{1}}\left\{v \in  \text{$k$-core of }G_n\left(\frac{1}{n}\right)\right\}
$$
If a vertex $v$ is in the $k$-core, then one of the two things must be true:
\begin{enumerate}[(i)]
\item $v$ is in a cycle within the $d$-neighborhood  of $v$ (this implies $v$ is in a cycle of length at most $2d$); 
\item Starting from $v$ there is a tree such that $v$ has $k$ neighbors, each of these $k$ neighbors has at least $k-1$ neighbors and this happens up to generation $d$. In this case we will call vertex $v$ has property $\mathcal{A}_d$.  
\end{enumerate}
Therefore 

\begin{align}\label{eq:upperbound}
C_k\left(G_n\left(\frac{1}{n}\right)\right) \leq & \sum_{v \in [n]} {\mathbbm{1}}\left\{v \text{  is in a cycle of length at most } 2d\right\} \notag \\ &+ \sum_{v \in [n]} {\mathbbm{1}}\left\{v \text{  has property  } \mathcal{A}_d\right\} \notag \\
=& \text{Term I}+\text{Term II}.
\end{align}
Let $V_n$ be an uniform random variable on $\{1,2,\ldots,n\}$ independent of everything else. Then according to our construction, 
\begin{align}\label{eqn:frimay8124pm}
\E ( \text{Term II}) &\leq n\P\left(^{*}X^{n} \text{  with root }V_n  \text{  has property  } \mathcal{A}_d\right) \notag \\
& \leq n\P\left(X^{n,\rho_n} \text{  with root }V_n  \text{  has property  } \mathcal{A}_d\right). 
\end{align}

Before presenting our first proposition, we state an auxiliary result, the BKR inequality (see e.g. \cite{MR1703130}). Consider a product space $\Omega$ of finite sets $\Omega_1, \dotso, \Omega_k$, 
\begin{align*}
\Omega = \Omega_1 \times \dotso \times \Omega_k. 
\end{align*}
Let $\mathcal{F}= 2^{\Omega}$, and $\mu$ be a product of $k$ probability measures $\mu_1, \dotso, \mu_k$. For any configuration $\omega=(\omega_1, \dotso, \omega_k) \in \Omega$, and any subset $S$ of $[k]:=\{1,\dotso,k\}$, we define the cylinder $[\omega]_S$ by 
\begin{align*}
[\omega]_S:= \{\hat{\omega}: \hat{\omega}_i = \omega_i, \,  \forall i \in S \}.
\end{align*}
For any two subsets $A, B \subset \Omega$, define 
\begin{align*}
A \circ B:= \{\omega: \text{there exists some } S=S(\omega) \subset [k] \text{ such that } [\omega]_S \subset A, \, [\omega]_{S^c} \subset B\}.
\end{align*}

\begin{lem}\label{lem:BK}
For any product space $\Omega$ of finite sets, product probability measure $\mu$ on $\Omega$ and $A, B \subset \Omega$, we have the inequality
\begin{align*}
\mu(A \circ B) \leq \mu(A) \mu(B).
\end{align*}
\end{lem}
In this paper, to apply the BKR inequality, we always take $\Omega^n_{i,j}=\{0,1\}, \,  i \not= j  \in \{1, \dotso, n\},$ and $\Omega = \prod_{i \not= j \in [n]} \Omega^n_{i,j}$.  Then $\omega^n_{i,j}=1$ represents that the node $i$ and $j$ are linked in the random graph $G_n \left(\frac{1}{n} \right)$. According to our construction, we also have $\mu_i(\{1\})= \min \{ a^n_{i,j}/n, 1\}$.

\begin{prop}\label{prop:frijun121237pm}
Let $G_n$ be a  sequence of graphs with non-negative edge weights which are bounded above by a constant $\bar{a}>0$. Then for any fixed $d$, it holds that 
$$C_k\left(G_n\left(\frac{1}{n}\right)\right) \leq n\P\left(X^{n,\rho_n} \text{  with root }V_n  \text{  has property  } \mathcal{A}_d\right)
+o_p(n).$$
\end{prop}
\begin{proof}
According to \eqref{eq:upperbound} and \eqref{eqn:frimay8124pm}, it suffices to show that $$\text{Term II}= \E(\text{Term II})+o_p(n), \quad \text{and} \quad \text{Term I}=o_p(n).$$ In the first two steps, we show the concentration of $\text{Term II}$ by computing its variance, and in the last step prove that $\text{Term I}$ is small. 

\vspace{6pt}
\noindent\textbf{Step I:} For any two independently and uniformly chosen vertices $U$ and $V$ of $G_n\left(\frac{1}{n}\right)$,  
$$
\P(d(U,V)\leq 2d) = \frac{1}{n^2}\sum_{u,v \in [n]}\P(d(u,v)\leq 2d) = o(1),
$$
where $d$ is the graph distance. To see this, note that $d(U,V)\leq 2d$ implies there is a path from $U$ to $V$ of length at most $2d$. Thus 
$$
\P(d(U,V)\leq 2d) \leq \sum_{i=1}^{2d}\P(\#\{\text{paths of length  }i\text{  from  }U\text{  to  } V\} \geq 1)
$$
Using Markov's inequality we get
$$
\P(d(U,V)\leq 2d) \leq \frac{1}{n^2}\sum_{u,v \in [n]}\sum_{i=1}^{2d}\E(\#\{\text{paths of length  }i\text{  from  }u\text{  to  } v\})
$$
We can get a crude upper bound as 
$$
\P(d(U,V)\leq 2d) \leq \frac{1}{n^2}\sum_{u,v \in [n]}\sum_{i=1}^{2d}n^{i-1}\left( \frac{\bar{a}}{n}\right)^i =o(1).
$$

\vspace{6pt}
\noindent\textbf{Step II:} 
Let $G^d_n[v]$ be the subgraph of $G_n\left( \frac{1}{n} \right)$ formed by the vertices within distance $d$ of $v \in [n]$, and define $B_{v}=\{  \text{root $v$ has property } \mathcal{A}_d \text { in $G^d_n[v]$}\}$. It can be easily verified that
\begin{align}\label{eqn:frimay8149pm}
\E({\text{Term II}}^2) =&  \sum_{v, v' \in [n]} \P\left(\text{root }v \text{  and  } v'  \text{  has property  } \mathcal{A}_d \right)  \notag \\
=& \sum_{ v \in [n]} \P(B_v)+ \sum_{ v \not = v'} \P(B_v \cap B_{v'})
\end{align}

For two different vertices $v$ and $v'$, we break the probability in two parts,
\begin{align}\label{eqn:monmay11114pm}
  \P (B_v \cap B_{v'})=&  \P(B_v \cap B_{v'}, d(v,v') \leq 2d)  + \P(B_v \cap B_{v'}, d(v,v')> 2d ).
\end{align}
For the second term on the right of  \eqref{eqn:monmay11114pm}, it can be easily seen that $$\{ d(v,v') > 2d\} \cap B_{v} \cap B_{v'} \subset \{ d(v,v') > 2d\}  \cap B_v \circ B_{v'}.$$
Therefore we get that 
\begin{align*}
\P(B_v \cap B_{v'}) &= \P\left( d(v,v')\leq 2d\right) 
+ \P\left(B_{v} \circ B_{v'} , d(v,v')> 2d\right)  \notag \\
& \leq \P\left( d(v,v')\leq 2d\right) + \P\left(B_{v} \circ B_{v'} \right).
\end{align*}
Now since $B_{v}$ and $B_{v'}$ are increasing events, according to Lemma~\ref{lem:BK} we obtain that 
\begin{equation}\label{eqn:monmay11131pm}
\P\left(B_{v} \cap B_{v'}\right)  \leq \P\left( d(v,v')\leq 2d\right) + \P(B_{v})\P(B_{v'}).
\end{equation}
Combining \eqref{eqn:monmay11131pm} and \eqref{eqn:frimay8149pm} we get 
\begin{equation}
\E({\text{Term II}}^2) \leq  n^2\P\left( d(U,V)\leq 2d\right) + \left(\E\left(\text{Term II}\right)\right)^2+\sum_{v \in [n]} \left( \P(B_v)-\P(B_v)^2 \right).
\end{equation}
Therefore using \textbf{Step I} we get $\V({\text{Term II}}) = o(n^2)$. Now using Markov's inequality we conclude that $\text{Term II}= \E(\text{Term II})+o_p(n)$.

\vspace{6pt}
\noindent\textbf{Step III:} Let us denote $C_{v} := \{v \text{  is in a cycle of length at most } 2d \}$. The first moment of the $\text{Term I}$ is given by
\begin{equation}
\sum_{v \in [n]} \P\left(C_v\right) \leq n \sum_{l=3}^{2d} \frac{(n-1)!}{(n-l)!} \left(\frac{\bar{a}}{n} \right)^{l} \leq \sum_{l=3}^{2d} \bar{a}^{l} =o(n).
\end{equation}

For the second moment, note that 
\begin{align*}
\E(\text{Term I}^2) =&\sum_{v \in [n]} \P(C_v)+  \sum_{v \not = v'} \P(C_v \cap C_{v'}).
\end{align*}
For two different vertices $v$, $v'$, the probability can be written as 
\begin{align*}
\P( C_v \cap C_{v'} ) = \P(C_v \cap C_{v'}, d(v,v') > 2d)+\P(C_v \cap C_{v'}, d(v,v') \leq 2d),
\end{align*}
and therefore 
\begin{align*}
\P( C_v \cap C_{v'} ) \leq \P(C_v \cap C_{v'}, d(v,v') > 2d)+\P( d(v,v') \leq 2d).
\end{align*}
Note that 
\begin{align*}
\{ d(v,v') > 2d\} \cap C_v \cap C_{v'}   \subset \{ d(v,v') > 2d\} \cap C_v \circ C_{v'}. 
\end{align*}
Therefore according to Lemma~\ref{lem:BK}, we obtain that 
\begin{align*}
\P(C_v \cap C_{v'} ) &\leq \P(C_{v} \circ C_{v'}, d(v,v') > 2d) +\P(d(v,v') \leq 2d) \\
&   \leq  \P(C_{v} \circ C_{v'}) +\P(d(v,v') \leq 2d)  \\
&\leq \P(C_{v})\P( C_{v'}) +\P(d(v,v') \leq 2d).
\end{align*}
Now summing over all $v, v' \in [n]$ and using \textbf{Step I}, we get

\begin{align*}
\E(\text{Term I}^2) = \E(\text{Term I})^2 +o(n^2) 
\end{align*}
We can conclude our result by using Markov's inequality.
\end{proof}

\subsection{Recursive formula}
Let us first introduce some notation. For any graphon $W$, we denote the initial particle of its associated branching process $X^W$ by $X^W_0$, and the first generation by $X^W_{\{1\}},\dotso,X^W_{\{N(W)_0\}}$, where $N(W)_0$ is the number of offsprings of $X^W_0$. For each element in the $d$-th generation, we denote it by $X^W_{\{i_1|i_2| \dotso |i_d\} }$ if he is the $i_d$-th child of $X^W_{\{i_1|i_2| \dotso |i_{d-1}\} }$. Denote the number of offsprings of $X^W_{\{i_1|i_2| \dotso |i_d\} }$ by $N(W)_{\{i_1|i_2| \dotso |i_d\} }$, and the type of $X^W_{\{i_1|i_2| \dotso |i_d\} }$ by $T(W)_{\{i_1|i_2| \dotso |i_d\} }$. Define the collection of offspring numbers in the first $d$ generations by
\begin{align*}
\nN(W)^d:=\{ N(W)_0\} \cup \dotso \cup \{ N(W)_{\{i_1|i_2| \dotso |i_d\}}:  i_j \leq N(W)_{\{i_1|i_2| \dotso |i_{j-1}\}}, j=1,\dotso,d \},
\end{align*}
and the collection of offspring numbers of $X^W_{\{i\}}$, $1\leq i \leq N(W)_0$ by 
\begin{align*}
\nN(W)^d_{\{i\}}:= \{N(W)_{\{i\}}\} \cup \dotso \cup \{ N(W)_{\{i|i_2| \dotso |i_d\}}:  i_j \leq N(W)_{\{i_1|i_2| \dotso |i_{j-1}\}}, j=2, \dotso, d\}.
\end{align*}
Denote the realizations of random variables $\nN(W)^d$ and $ \nN(W)^d_{\{i\}}$ by $\kK^d$ and $ \kK^d_{\{i\}}$ respectively, and especially denote the realization of $N(W)_0$ by $k_0$. Define functions 
\begin{align*}
g(x, \kK^d):=\P( \nN(W)^d= \kK^d\,|\, T(W)_0 =x).
\end{align*}
It is clear that 
\begin{align*}
\P( \nN(W)^d= \kK^d)=\int g(x, \kK^d) \, dx.
\end{align*}

\begin{prop}\label{prop:recursive2}
We have that 
\begin{align}\label{eq:recursive2}
g(x, \kK^d)= \frac{e^{-\int W(x,y)  dy}}{ k_0!} \prod_{j=1}^{k_0} \left( \int W(x,y)  g(y,\kK^d_{\{j\}}) \, dy\right).
\end{align}
\end{prop}
\begin{proof}
It can be easily seen that $g(x, k_0)=\frac{1}{k_0!}e^{-\int W(x,y) dy }\left(\int W(x,y) dy \right)^{k_0}$. For $d \geq 1$, we get that 
\begin{align*}
g(x,\K^d)=& \P(\nN(W)^d =\kK^d\,| \,T(W)_0=x) \\
=&\P(N(W)_0=k_0\,|\, T(W)_0=x) \\
& \times \P(\nN(W)^d_{\{j\}}=  \kK^d_{\{j\}}, j=1, \dotso,k_0 \,|\,N(W)_0=k_0,T(W)_0=x    ) \\
= &g(x,k_0) \int_{y_1} \dotso \int_{y_{k_0}} \prod_{j=1}^{k_0} g(y_j, \kK^d_{\{j\}}) \, \P(T(W)_{\{j\}} \in dy_j\,|\, N(W)_0=k_0, T(W)_0=x) .
\end{align*}
In conjunction with the equation
\begin{align*}
\prod_{j=1}^{k_0} \P(T(W)_{\{j\}} \in d y_j\,|\, N(W)_0 =k_0, T(W)_0=i) = \frac{\prod_{j=1}^{k_0} W(x, y_j) \, dy_j }{ (\int_y W(x,y) \, dy)^{k_0} },
\end{align*}
we obtain the recursive formula
\begin{align*}
g(x, \kK^d)= \frac{e^{-\int W(x,y) dy}}{ k_0!} \prod_{j=1}^{k_0} \left( \int W(x,y)  g(y,\kK^d_{\{j\}}) \, dy\right).
\end{align*}
\end{proof}

\subsection{Convergence}
Let $W_n$ be a sequence of graphons such that $d_\square(W_n,W) \to 0$  and $$\sup_{n,x,y} W_n(x,y) \leq \bar{a}$$ for some positive constant $\bar{a}$. Let $X^n$ be the associated branching process of $W_n$, and $$g_n(x,\kK^d)=\P( \nN(W_n)^d= \kK^d\,|\, T(W_n)_0 =x).$$

We want to show that as $n \to \infty$
\begin{align*}
\int g_n(x,\kK^d) \, dx \to \int g(x,\kK^d) \, dx. 
\end{align*}
To see this, for any graphon $W$, any finite tree $T$ with root $0$, any $x \in [0,1]$, we define the vertex prescribed homomorphism density
\begin{align*}
t^x(T,W)= \int_{[0,1]^{|V(T)|-1} } \prod_{0i \in E(T)} W(x,x_i) \prod_{ij \in E(T), i,j \geq 1}W(x_i,x_j)  \, dx_1 \dotso dx_{|V(T)|-1},
\end{align*}
and the homomorphism density
\begin{align*}
t(T,W)=\int_{[0,1]} t^x(T,W) \, dx.
\end{align*}
It is well-known that for finite $T$, $t(T,W_n) \to t(T,W)$ as long as $d_\square(W_n,W) \to 0$; see e.g. \cite{borgs2008convergent, borgs2012convergent,lovasz2006limits}. We will rewrite $\int g_n(x,\kK^d) \, dx $ and $\int g(x, \kK^d) \,dx)$ as $\sum_{m \geq 0} \lambda_m t(T_m,W_n)$ and $\sum_{m \geq 0} \lambda_m t(T_m,W)$ respectively for a sequence of trees $T_m$. 

\begin{prop}\label{prop:tree}
Suppose $W$ is a graphon such that $\sup_{x,y} W(x,y) \leq \bar{a}$. Then for any $d \in \N$ and any configuration $\kK^d$, there exists a sequence of finite trees $(T_m)_{m \geq 0}$, and a sequence of real numbers $(\lambda_m)_{m \geq 0}$ such that 
\begin{enumerate}[(i)]
\item $\sum_{m \geq 0} |\lambda_m| \, \bar{a}^{|E(T_m)|}<+\infty$ ;
\item $g(x,\kK^d)=\sum_{m \geq 0} \lambda_m \, t^x(T_m,W)$.
\end{enumerate}
\end{prop}
\begin{proof}
Let us prove by induction. For $d=0$, we have that 
\begin{align*}
g(x,k)= \frac{1}{k!}e^{-\int W(x,y) dy }\left(\int W(x,y) dy \right)^k= \frac{1}{k!} \sum_{m=0} \frac{(-1)^m}{m!} \left(\int W(x,y) dy \right)^{m+k}.
\end{align*}
For any $m \in \N$, take $T_m$ to be an $(m+k)$-star, i.e., a tree of height $1$ with $(m+k)$ leaves. Define $\lambda_m: =\frac{(-1)^m}{k!m!}$. Then it can be easily seen that $$\sum_{m \geq 0} |\lambda_m| \, \bar{a}^{|E(T_m)|}=\sum_{m \geq 0} \frac{\bar{a}^{t+k}}{k!m!} <+\infty,$$ and $$g(x,k)=\sum_{ m \geq 0} \lambda_m \,t^x(T_m,W). $$
Now suppose that our claim is true for any configuration $\kK^{d-1}$. According to our recursive formulas  \eqref{eq:recursive2}, we expand the exponential term and obtain that 
\begin{align*}
g(x, \kK^d)&= \frac{1}{k_0!} \sum_{m \geq 0} \frac{(-1)^m}{m!} \left(\int W(x,y) dy\right)^m\prod_{j=1}^{k_0} \left( \int W(x,y) \, g(y,\kK^d_{\{j\}}) \, dy\right).
\end{align*}

\begin{figure}

\begin{tikzpicture}
  \node {root}
    child {node {$1$}}
    child {node {$\dotso$}}
    child {node{$m_0$}}
    child{ node {$T^1_{m_1}$}}
    child {node {$\dotso$}}
    child{ node {$T^{k_0}_{m_{k_0}}$}};
\end{tikzpicture}
\caption{Tree $T_m$}
\label{fig:tree}
\end{figure}
For each $\kK^d_{\{j\}}, j=1,\dotso, k_0$, we have sequences $(\lambda_m^j)_{m \geq 0}, (T^j_m)_{m \geq 0}$ such that our claim is satisfied. For each $m=(m_0, m_1, \dotso, m_{k_0}) \in \mathbb{N}^{k_0+1}$, we define $\lambda_m=\frac{(-1)^{m_0}}{k_0! m_0!} \prod_{j=1}^{k_0} \lambda_{m_j}^{j},$ and tree $T_m$ as in Figure~\ref{fig:tree}. It is then clear that 
\begin{align*}
\sum\limits_{m \in \N^{k_0+1}} |\lambda_m| \, \bar{a}^{|E(T_m)|} \leq  \sum_{m_0 \in \N} \frac{\bar{a}^{k_0+m_0}}{k_0! m_0!} \prod_{j=1}^{k_0} \left(\sum_{m_j \in \N} |\lambda_{m_j}^{j}| \, \bar{a}^{|E(T_{m_j}^j)|} \right) < +\infty.
\end{align*}
According to our induction, we have that 
\begin{align*}
g(y,\kK^d_{\{j\}}) = \sum_{m_j \geq 0} \lambda_{m_j}^{j}t^y(T^j_{m_j},W). 
\end{align*}
Therefore, we obtain that 
\begin{align*}
g(x, \kK^d)&=\sum\limits_{m \in \N^{k_0+1}} \lambda_m \left(\int W(x,y) dy\right)^{m_0} \prod_{j=1}^{k_0} \left(\int W(x,y)\, t^y(T_{m_j}^j, W) \,dy \right).
\end{align*}
It can be easily verified that for each $m \in \N^{k_0+1}$,
\begin{align*}
t^x(T_m, \kK^d) =\left(\int W(x,y) dy\right)^{m_0} \prod_{j=1}^{k_0} \left(\int W(x,y)\, t^y(T_{m_j}^j, W) \,dy \right).
\end{align*}
Thus, we conclude that $$ g(x,\kK^d)= \sum_{m \in \N^{k_0+1}} \lambda_m \, t^x(T_m,W).$$
\end{proof}

\begin{prop}\label{prop:convergence} Suppose  $W_n$ is a sequence of graphons such that $d_\square(W_n,W) \to 0$, and satisfying $\sup_{n,x,y} W_n(x,y) \leq \bar{a}$ for some positive constant $\bar{a}$. Then it holds that
\begin{align*}
\lim\limits_{n \to \infty} \P(\nN(W_n)^d=\kK^d)=\P(\nN(W)^d=\kK^d) .
\end{align*}
\end{prop}
\begin{proof}
According to Proposition~\ref{prop:tree}, we get that
\begin{align*}
& \P(\nN(W_n)^d=\kK^d)=\int g_n(x, \kK^d)\, dx= \sum_{m \geq 1} \lambda_m \, t(T_m,W_n), \\
& \P(\nN(W)^d=\kK^d) =\int g(x, \kK^d) \, dx = \sum_{m \geq 1} \lambda_m \, t(T_m,W).
\end{align*}
Since $W_n$ converges to $W$ in that cut norm, we have that $t(T_m,W_n) \to t(T_m, W)$ as $n \to \infty$. Due to the uniform bound $$\sum_{m \geq 1} \lambda_m \, t(T_m,W_n) \leq \sum_{m \geq 1} \lambda_m \, \bar{a}^{|E(T_m)|} < +\infty,$$
we apply the dominated convergence theorem, and conclude that $\P(\nN(W_n)^d=\kK^d)$ converges to $\P(\nN(W)^d=\kK^d) $ as $n \to \infty$. 

\end{proof}

\subsection{Tightness}
Notice that $X^W \in \mathcal{A}_d$ is equivalent to that $\nN(W)^d \in \mathcal{A}_d$. To make our computation clear, we will sometimes adopt the latter notation. Recall we want to show that 
\begin{align}
\P(\nN(W)^d \in \mathcal{A}_d)=\lim\limits_{n \to \infty} \P(\nN(W_n)^d \in \mathcal{A}_d).
\end{align}
To apply Proposition~\ref{prop:convergence}, we need a tightness result. 
\begin{lem}\label{lem:tightness}
For $K \in \N$, we define $\nN(W)^d \leq K$ if $N(W)_{\{i_1|i_2|\dotso|i_j\}} \leq K$ for any  $X^W_{\{i_1|i_2|\dotso|i_j\}}$ in the first $d$ generations. Suppose $\sup_{x,y}W(x,y) \leq \bar{a}$ for some positive constant $\bar{a}$. Then for any $\alpha>0$, $d \in \N$, there exists a large enough $K_0 \in \N$ uniformly for $x \in [0,1]$ such that $K \geq K_0$ implies
\begin{align}\label{eq:tightness1}
\P (\nN(W)^d \leq K \, | \, T(W)_0 =x ) > 1- (1/K)^{\alpha}. 
\end{align}
Here, the choice of $K_0$ only depends on $\alpha$, $d$ and $\bar{a}$.
\end{lem}
\begin{proof}
Let us prove \eqref{eq:tightness1} by induction. Recall for the initial generation we have that 
$$g(x, k)=\frac{1}{k!}e^{-\int W(x,y) dy }\left(\int W(x,y) dy \right)^{k}.$$
For any $k \in \N$, we define for $c \in \mathbb{R}_+$
\begin{align*}
\psi_{k}(c) := \sum_{l=k+1}^{\infty} \frac{1}{l!} e^{-c} c^l. 
\end{align*}
Thus we have that 
\begin{align*}
\P( N(W)_0 \leq k \, | \, T(W)_0=x)=1- \psi_k \left(\int W(x,y) dy \right).
\end{align*}
It can be easily verified that $\psi_k'(c)=\frac{e^{-c} c^{k}}{k!} \geq 0$, and hence $\psi_k \left(\int W(x,y) dy \right)  \leq \psi_k(\bar{a})$. Take $K$ large enough that $\psi_K(\bar{a}) < (1/K)^{\alpha}$. Then it is clear that 
\begin{align*}
\P(N(W)_0  \leq K \, | \, T(X^n)_0=x) =1- \psi_K \left(\int W(x,y) dy \right)  >1- (1/K)^{\alpha}.
\end{align*}

Assume our claim is true for $d-1$. Then for any $\beta >0$, there exists a $K$ such that $$\P(\nN(W)^d_{\{j\}} \leq K \, | \, T(W)^d_{\{j\}}=y) \geq 1- (1/K)^{\beta} .$$
Note that 
\begin{align*}
& \P(\nN(W)^d \leq K \, | \, T(W)_0=x ) \\
& = \sum_{k=0}^{K} \P(\nN(W)^d_{\{j\}} \leq K, j=1, \dotso, k \, | \, N(W)_0=k, T(W)_0=x ) \, \P(N(W)_0=k \, | \, T(W)_0=x). 
\end{align*}
As in the proof of Proposition~\ref{prop:recursive2}, we have that 
\begin{align*}
\P(\nN(W)^d \leq K \, | & \, T(W)_0=x ) \\
=& \sum_{k=0}^{K} \left( \frac{e^{-\int W(x,y) dy}}{ k!} \prod_{j=1}^{k} \left( \int_y W(x,y) \P(\nN(W)^d_{\{j\}} \leq K \, | \, T(W)_{\{j\}}=y) \, dy \right) \right) \\
>& \sum_{k=0}^{K} \left( \frac{e^{-\int W(x,y) dy}}{ k!} \prod_{j=1}^{k} \left( \int_y   W(x,y) (1-(1/K)^{\beta}\, dy )\right) \right) \\
=& \sum_{k=0}^{K} \left( \frac{e^{-\int W(x,y) dy}}{ k!} \left( \int W(x,y) dy \right)^k (1-(1/K)^{\beta})^k \right).
\end{align*}

Since $(1-(1/K)^{\beta})^K > 1- (1/K)^{\beta-2}$ for large $K$, we have that 
\begin{align*}
\P(\nN(W)^d \leq K \, | \, T(W)_0=x ) >& \sum_{k=0}^{K} \left( \frac{e^{-\int W(x,y) dy}}{ k!} \left(\int W(x,y) dy \right)^k  \right) (1- (1/K)^{\beta -2}) \\
>& (1-\psi_K(\bar{a}))(1- (1/K)^{\beta -2}).
\end{align*}
Therefore by taking $\beta =\alpha+3$, and large $K$ such that $\psi_K(\bar{a}) < (1/K)^{\alpha+1}$, we conclude that 
\begin{align*}
\P (\nN(W)^d \leq K \, | \, T(W)_0 =x ) > 1- (1/K)^{\alpha}. 
\end{align*}

\end{proof}

\begin{prop}\label{prop:branchingeq}
 Suppose  $W_n$ is a sequence of graphons such that $d_\square(W_n,W) \to 0$, and satisfying $\sup_{n,x,y} W_n(x,y) \leq \bar{a}$ for some positive constant $\bar{a}$. Then for any fixed $d$, we have that 
\begin{align*}
\lim\limits_{n \to \infty} \P(\nN(W_n)^d \in \mathcal{A}_d) =\P(\nN(W)^d \in \mathcal{A}_d),
\end{align*}
from which we conclude that 
\begin{align*}
\limsup \limits_{n \to \infty} \P_{X^{W_n}}(\mathcal{A}) \leq \P_{X^W}(\mathcal{A}).
\end{align*}

\end{prop}
\begin{proof}
Due to Proposition~\ref{prop:convergence}, it can be seen that for fixed $d, K$ 
\begin{align*}
\lim\limits_{n \to \infty} \P(\nN(W_n)^d \in \mathcal{A}_d, \nN(W_n)^d \leq K) =\P(\nN(W)^d \in \mathcal{A}_d, \nN(W)^d \leq K).
\end{align*}
 Applying Lemma~\ref{lem:tightness}, we let $K \to \infty$, and obtain that 
\begin{align*}
\lim\limits_{n \to \infty} \P(\nN(W_n)^d \in \mathcal{A}_d)=\P(\nN(W)^d \in \mathcal{A}_d).
\end{align*}
  
For any $\epsilon >0$, there exists a $d$ such that 
\begin{align*}
\P(\nN(W)^d \in \mathcal{A}_d)=\P(X^W \in \mathcal{A}_d) \leq \P_{X^W}(\mathcal{A})+\epsilon.
\end{align*}
Then, it can be easily verified that 
\begin{align*}
\limsup \limits_{n \to \infty} \P_{X^{W_n}}(\mathcal{A}) \leq 
\limsup \limits_{ n \to \infty}  \P(\nN({W_n})^d \in \mathcal{A}_d)=\P(\nN(W)^d \in \mathcal{A}_d) \leq 
\P_{X^W}(\mathcal{A})+\epsilon.
\end{align*}
Therefore we obtain that
\begin{align*}
\limsup \limits_{n \to \infty} \P_{X^{W_n}}(\mathcal{A}) \leq \P_{X^W}(\mathcal{A}).
\end{align*}

\end{proof}

\subsection{Completing the proof of the upper bound}

Recalling Proposition~\ref{prop:frijun121237pm}, we have that 
$$ C_k\left(G_n \left( \frac{1}{n} \right) \right)  \leq n \P(X^{n,\rho_n} \in \mathcal{A}_d)+ o_p(n).$$
Note that $X^{n, \rho_n}$ is the branching process associated with the graphon $\rho_n W_{G_n}$, and $d_\square ( W_{G_n}, W) \to 0$. Applying Proposition~\ref{prop:branchingeq} with $W_n = \rho_n W_{G_n}$,  we obtain that $$C_k\left(G_n \left( \frac{1}{n} \right) \right) \leq n \P_{X^W}(\mathcal{A}_d) + o_p(n).$$
Letting $d \to \infty$ in the above inequality, we conclude our result.

$\hfill\square$

\section{The proof of the lower bound in Theorem~\ref{thm:554pm07dec20}}\label{sec:446pm10dec20}
Before proceeding to our proof of the lower bound, we argue that it suffices to prove it for irreducible $W$. A graphon $W$ is said to irreducible if there is no measurable $A \subset [0,1]$ such that $\text{Leb}(A) \in (0,1)$ and $W=0$ a.e. on $A\times A^c$.

 According to \cite[Lemma 5.17]{MR2337396}, for any graphon $W$, there exists a partition $[0,1]=\cup_{i=1}^N I_i$ with $0 \leq N \leq \infty$ such that $|I_i| >0$ for each $i \geq 1$, the restriction of $W$ on $I_i \times I_i$ is irreducible for $i \geq 1$, and $W=0$ a.e. on $([0,1] \times [0,1]) \setminus \cup_{i=1}^N(I_i \times I_i)$. Without loss of generality, we assume that for each $i\geq 1$, $I_i$ is an interval. Denote the restriction $W|_{I_i \times I_i}$ by $W_i$. To properly define the branching process of $W_i$, we identify it with a graphon $\widetilde{W}_i:I \times I \to [0,\infty)$ such that $\widetilde{W}_i(x,y)=W(x,y)$ for $x,y \in I_i$ and $\widetilde{W}_i(x,y)=0$ otherwise, i.e., $X^{\lambda W_i}:=X^{\lambda \widetilde{W}_i}$. Then it can be easily verified that $\P_{X^{\lambda W}}(\mathcal{A})= \sum_{i=1}^N \P_{X^{\lambda W_i}}(\mathcal{A})$. Hence under Assumption~\ref{assume2}, $\lambda \mapsto \P_{X^{\lambda W_i}}(\mathcal{A})$ is continuous at $\lambda=1$ for each $i \geq 1$. Fix $N_0 \leq N$. Each $G_n$ can be decomposed into $N_0$ disconnected subgraphs $(G_n^i)_{i=1,\dotso,N_0}$ such that $G_n^i$ converges to $W|_{I_i \times I_i}$ in the cut norm. \emph{Assuming our result holds for irreducible graphons}, we obtain that (after appropriate normalization of $G_n^i$ and $W_i$) 
\begin{align*}
C_k\left(G^i_n \left( \frac{1}{n} \right) \right) \geq n  \P_{X^{W_i}}(\mathcal{A}_d) + o_p(n),
\end{align*}
and thus
\begin{align*}
C_k\left(G_n \left( \frac{1}{n} \right) \right) \geq \sum_{i=1}^{N_0} C_k\left(G^i_n \left( \frac{1}{n} \right) \right) \geq n  \sum_{i=1}^{N_0} \P_{X^{W_i}}(\mathcal{A}_d) + o_p(n).
\end{align*}
Letting $N_0 \to N$, we can conclude the proof for the lower bound. From now on, let us assume that \emph{$W$ is irreducible}.

We say a graphon $F$ is finitary if there exist finitely many disjoint intervals $I_{t_i}, i=1, \dotso, M$ such that  $\cup_{i=1}^M I_{t_i}=[0,1]$ and the restriction of $F$ on $I_{t_i} \times I_{t_j}$ is a constant for any $1 \leq i, j \leq M$. According to \cite[Lemma 7.3]{MR2337396}, the graphon $W$ can be approximated pointwise from below by finitary graphons. More precisely, we have that 
\begin{lem}\label{lem:approx}
 There exists a sequence of finitary graphons $(F_m)_{m \in \N}$ such that $F_m \leq W$ and $\lim\limits_{m \to \infty} F_m(x,y)=W(x,y)$ a.s. 
\end{lem}

Under Assumption~\ref{assume2}, we will prove in Subsection~\ref{subsec:main1} that for any $\vep>0, m \in \N$, 
\begin{align}\label{eq:main1}
 C_k\left(G_n\left(\frac{1}{n}\right)\right) \geq (1-2\vep) n \P_{X^{(1-2\vep)F_m}}(\mathcal{A})+o_p(n).
\end{align}
Then in Subsection~\ref{subsec:main2}, we will show the continuity property 
\begin{align}\label{eq:main2}
\liminf\limits_{\vep \to 0, m \to \infty}  \P_{X^{(1-2\vep)F_m}}(\mathcal{A}) \geq \P_{X^W}(\mathcal{A}).
\end{align}
It is clear that \eqref{eq:main1} and \eqref{eq:main2} together prove the lower bound part of Theorem~\ref{thm:554pm07dec20}, i.e., 
$$C_k\left(G_n \left( \frac{1}{n} \right) \right) \geq n \P_{X^W}(\mathcal{A}) + o_p(n).$$

\subsection{Proof of \eqref{eq:main1}}\label{subsec:main1}

Fixing $m \in \N$, and $\vep \in (0,\frac{1}{m})$ such that  $\lambda \to \P_{X^{\lambda (1-\vep)F_m}}(\mathcal{A})$ is continuous at $\lambda =1$.  Suppose $[0,1]$ is a disjoint union of intervals $I_{t_j}, j=1 ,\dotso, M$, and there exists a collection $\{F_m(t_i,t_j): \, 1\leq i,j \leq M\}$ such that $F_m(x,y)=F_m(t_j,t_k)$ for $ x \in I_{t_j}, y \in I_{t_k}$. Here we say $t_h, h=1,\dotso,M$ labels to distinguish types in the definition of branching process $X^n$. 

Before proceeding to the rigorous proof, let us first give main ideas of our argument. We divide vertices of $G_n$ into $M$ groups $\mathbf{Good}_{n, t_1},\dotso, \mathbf{Good}_{n, t_M}$ with the property that for any vertex $i \in \mathbf{Good}_{n, t_h}$ and $k=1,\dotso, M$
\begin{align*}
\frac{\widetilde{d}^n_{i,t_k}}{n}:= \sum\limits_{ j\in \mathbf{Good}_{n,t_k}} \frac{a^n_{i,j}}{n} \geq (1-\vep)F_m(t_h,t_k)|I_k|. 
\end{align*}
Therefore, we can heuristically consider $G_n$ as a `finitary' graph by labelling vertices in $\mathbf{Good}_{n,t_h}$ by $t_h$, $h=1,\dotso, M$. Due to the above inequality, the branching process $X^{n}$ associated with $G_n$ stochastically dominates, in the first order, the branching process $X^{(1-\vep)F_m}$ associated with $(1-\vep)F_m$. Take $F^{\vep}_m\left( \frac{1}{n}\right)$ to be an $n$-vertex random graph sampled from $(1-\vep)F_m$, i.e., independently uniformly select vertices $v_i \in [0,1]$ and then connect $v_i,v_j$ independently with probability $(1-\vep)F_m(v_i,v_j)/n$. By the standard exploration argument (see e.g. \cite[Section 9]{MR2337396}), locally the random graph $G_n\left( \frac{1}{n} \right)$ ($F^{\vep}_m\left( \frac{1}{n} \right)$ resp.) is almost the branching process $X^{G_n}$ ($X^{(1-\vep)F_m}$ resp.). Thus, heuristically the random graph  $G_n\left( \frac{1}{n} \right)$ is more connected than the random graph $F^{\vep}_m\left( \frac{1}{n}\right)$, and thus has larger size of $k$-core.  Therefore, the inequality \eqref{eq:main1} follows from \cite[Theorem 3.1]{MR2376426}, which says $$C_k\left(F^{\vep}_m\left( \frac{1}{n}\right) \right)= n \P_{X^{(1-\vep)F_m}}(\mathcal{A})+o_p(n) .$$ The following simple lemma will be used to label vertices of $G_n$.  
 
\begin{lem}\label{lem:approx}
Suppose $\vep >0$ is a fixed constant, and $\|W_{G_n}-W\|_{\square} \to 0$.  Then there exists a collection of disjoint subsets $\widetilde{\mathbf{Bad}}_{n,t_j} \subset I_{t_j}, j=1,\dotso, M$ such that 
\begin{enumerate}[(i)]
\item $ |\widetilde{\mathbf{Bad}}_{n,t_j}| = o(1), \, j=1,\dotso, M$.
\item For any $x \in  I_{t_j} \setminus \widetilde{\mathbf{Bad}}_{n,t_j}$, we have that 
\begin{align}\label{eq:include3}
\int_{I_{t_k}} W_{G_n}\left(x, y\right) \, dy \geq  (1-\vep/2)  F_m(t_j,t_k) |I_{t_k}|, \, \, \,   k=1,\dotso, M.
\end{align}
\end{enumerate}  
\end{lem}
\begin{proof}
First let us recall that 
 one can also write
\begin{equation}\label{eqn:jul9415pm}
\|W\|_\square = \sup_{0\leq f,g \leq1 \text{   measurable  }} \left| \int f(x) g(y) W(x,y)\,dx\,dy \right|.
\end{equation}
For any $1 \leq j,k \leq M$, define
\begin{align*}
\widetilde{\mathbf{Bad}}_{n,t_j,t_k}=\left\{ x \in I_{t_j}:  \int_{I_{t_k}} W_{G_n}\left(x, y\right) \, dy <    (1-\vep/2)  F_m(t_j,t_k) |I_{t_k}| \right\}.
\end{align*}
In the case that $F_m(t_j,t_k) =0$, it is clear that $\widetilde{\mathbf{Bad}}_{n,t_j,t_k}= \emptyset$. Let us assume that $F_m(t_j,t_k)>0$.

Taking $f(x)=\mathbbm{1}_{\{x \in \widetilde{\mathbf{Bad}}_{n,t_j,t_k}\}}$ and $g(y)=\mathbbm{1}_{\{y \in I_{t_k}\}}$ in \eqref{eqn:jul9415pm} , we obtain that 
\begin{align*}
\int_{\widetilde{\mathbf{Bad}}_{n,t_j,t_k}} \, dx \int_{I_{t_k}} (W_{G_n}(x,y)- W(x,y)) \, dy \geq -\|W_{G_n}-W\|_{\square}=-o(1),
\end{align*}
and due to $F_m \leq W$, 
\begin{align}\label{eq:lowerbound2}
\int_{\widetilde{\mathbf{Bad}}_{n,t_j,t_k}} \, dx \int_{I_{t_k}} (W_{G_n}(x,y)- F_m(x,y)) \, dy \geq -o(1). 
\end{align}
Since for any $x \in \widetilde{\mathbf{Bad}}_{n,t_j,t_k}$ $$\int_{I_{t_k}} (W_{G_n}(x,y)- F_m(x,y)) \, dy < - \vep/2 F_m(t_j,t_k) |I_{t_k}|,$$ together with \eqref{eq:lowerbound2} it follows that 
\begin{align}\label{eq:lowerbound3}
\vep F_m(t_j,t_k) |I_{t_k}| | \widetilde{\mathbf{Bad}}_{n,t_j,t_k} | \leq  o(1).
\end{align} 
As a result of $F_m(t_j,t_k)>$, we obtain that 
\begin{align*}
|\widetilde{\mathbf{Bad}}_{n,t_j,t_k}| =o(1).
\end{align*}

Let us take $$\widetilde{\mathbf{Bad}}_{n, t_j}=\bigcup\limits_{k=1}^M \widetilde{\mathbf{Bad}}_{n,t_j,t_k},$$
and it is clear that $|\widetilde{\mathbf{Bad}}_{n, t_j}| = o(1)$ satisfies \eqref{eq:include3}. 
\end{proof}

Before proving the main result in this subsection, we would like to point out that our main contribution here is the observation that one can label vertices of $G_n$ so that heuristically it dominates  the finitary graphon $(1-\vep)F_m$. The remaining part of proof is just a modification of \cite[Theorem 3.1]{MR2376426}. We summarize it as the following lemma, and refer the reader to \cite{MR2376426} for a detailed argument.

\begin{lem}\label{lem:riordan}
Suppose $F_m$ is a finitary graphon with $M$ labels $t_1, \dotso, t_M$, and  $\lambda \to \P_{X^{\lambda F_m}}(\mathcal{A})$ is continuous at $\lambda =1$. Let $G_n$ be a sequence of graphs such that $\sup \{ a^n_{i,j} \} < +\infty$. Denote by $X^{G_n}_i$ ($X^{F_m}_{t_h}$ resp.) the branching process associated with $G_n$ ($F_m$ resp.) that has the initial particle with type $i$ (label $t_h$ resp.). If the vertices of $G_n$ can be divided into $M$ groups $\mathbf{G}_{n,t_h}, h=1, \dotso, M$ such that for some $\vep \in (0,1)$
\begin{enumerate}[(i)]
\item $\frac{|\mathbf{G}_{n, t_h}|}{n} \geq (1-\vep)|I_{t_h}|,  h=1, \dotso, M$,
\item For each vertex $i \in \mathbf{G}_{n, t_h}$, the branching process $X^{G_n}_{i}$ stochastically dominates, in the first order, the branching process $X^{F_m}_{t_h}$,
\end{enumerate}
then it holds that 
\begin{align*}
C_k\left(G_n\left( \frac{1}{n} \right) \right) \geq (1-\vep)n \P_{X^{(1-\vep)F_m}}(\mathcal{A})+ o_p(n). 
\end{align*}
\end{lem}

\begin{proof}[Completing the proof of \eqref{eq:main1}]
Since $\lambda \to \P_{X^{\lambda F_m}}(\mathcal{A})$ is non-decreasing with respect to $\lambda$, it has only countably many discontinuity points. Therefore we can choose arbitrarily small $\vep$ such that $\lambda \to \P_{X^{\lambda(1-\vep)F_m}}(\mathcal{A})$ is continuous at $\lambda=1$. We choose $0<\vep<\frac{1}{m}$, and take $\widetilde{\mathbf{Bad}}_{n,t_h}$ as in Lemma~\ref{lem:approx}.  For $h=1, \dotso, M$, define 
\begin{align*}
\mathbf{Good}_{n, t_h}: =\left\{ i  \in [n]: \left( \frac{i-1}{n}, \frac{i}{n} \right] \in I_{t_h} \setminus \widetilde{\mathbf{Bad}}_{n,t_h} \right\}. 
\end{align*}
Due to the construction of $W_{G_n}$ in \eqref{eqn:satapr11348pm}, for any $\left( \frac{i-1}{n}, \frac{i}{n} \right] \in I_{t_h}$, we have either $ \left( \frac{i-1}{n}, \frac{i}{n} \right] \subset \widetilde{\mathbf{Bad}}_{n,t_h}$ or $\left( \frac{i-1}{n}, \frac{i}{n} \right] \cap \widetilde{\mathbf{Bad}}_{n,t_h}=\emptyset$. Therefore it can be easily verified that 
\begin{align*}
\frac{|\mathbf{Good}_{n, t_h}|}{n}  \geq |I_{t_h}|- o(1).
\end{align*}
For any $i \in \mathbf{Good}_{n,t_h}$, define 
\begin{align*}
\widetilde{d}^n_{i,t_k}:= \sum\limits_{ j\in \mathbf{Good}_{n,t_k}}a^n_{i,j}, \quad k=1,\dotso, M.
\end{align*}
As a result of \eqref{eq:include3}, we obtain that 
\begin{align*}
\frac{\widetilde{d}^n_{i,t_k}}{n} &\geq \int_{I_{t_k}} W_{G_n}\left(i/n, y\right) \, dy - \bar{a} \left(|I_{t_k}|- \frac{|\mathbf{Good}_{n, t_k}|}{n}  \right)\\
&\geq (1-\vep/2)  F_m(t_h,t_k) |I_{t_k}|-o(1)\bar{a}.
\end{align*}
We conclude that for large enough $n$, there exists a collection of disjoint $\mathbf{Good}_{n,t_h} \subset [n], h=1, \dotso, M$, which satisfies the following
\begin{enumerate}[(i)]\label{enm:bad}
\item For all $ h=1, \dotso, M$, 
\begin{align}\label{eq:include2}
\frac{|\mathbf{Good}_{n,t_h}|}{n} \geq (1- \vep)|I_{t_h}|.
\end{align}
\item For any $i \in \mathbf{Good}_{n,t_h}$, it holds that 
\begin{align}\label{eq:include}
\frac{\widetilde{d}^n_{i,t_k}}{n} \geq (1-\vep) F_m(t_h,t_k)|I_{t_k}|, \quad k=1,\dotso, M.
\end{align}
\end{enumerate}

For vertices in $\mathbf{Good}_{n, t_h}$, $h=1,\dotso,M$, we label them with $t_h$. Let us define $$\mathbf{Good}_{n}:= \bigcup_{h=1}^M \mathbf{Good}_{n,t_h}, \quad \tilde{n}:=|\mathbf{Good}_{n}| .$$
Let $\widetilde{G}_{\tilde{n}}$ be a graph with vertices  $\mathbf{Good}_n$ such that $\widetilde{a}^n_{i,j}:=\tilde{n}a^n_{i,j} /n$ for all $i,j \in \mathbf{Good}_n$. 
It is clear that 
\begin{align}\label{eq:conclude1}
C_k\left(G_n\left(\frac{1}{n} \right) \right) \geq C_k\left(\widetilde{G}_{\tilde{n}}\left(\frac{1}{\tilde{n}} \right) \right) .
\end{align}

Take $\widetilde{X}^{\tilde{n}}$ to be a branching process sampled from $\widetilde{G}_{\tilde{n}}$. For any $i \in \mathbf{Good}_{n,t_h}$, take $\widetilde{X}^{\tilde{n}}_i$ to be a branching process sampled from graph $\widetilde{G}_{\tilde{n}}$ with root $i$. For any $t_h$, take $X^{(1-\vep)F_m}_{t_h}$ to be a branching process sampled from kernel $(1-\vep)F_m$ with root of label $t_h$. Suppose a particle in generation $t$ of $\widetilde{X}_i^{\tilde{n}}$ is of type $j$ with label $t_h$, as a result of \eqref{eq:include} the number of its $t_k$-labelled children has Poisson distribution with parameter $\widetilde{d}^n_{j,t_k}$ larger than $(1-\vep) F_m(t_h,t_k)|I_{t_k}|$. Therefore, for any $i \in \mathbf{Good}_{n,t_h}$, we can consider $X^{(1-\vep)F_m}_{t_h}$ as a subset of $\widetilde{X}_i^{\tilde{n}}$. Thus for any increasing event $\mathcal{I}$, we have that $\P_{\widetilde{X}_i^{\tilde{n}}}(\mathcal{I}) \geq \P_{X^{(1-\vep)F_m}_{t_h}}(\mathcal{I})$, and also
\begin{align}\label{eq:important}
\P_{\widetilde{X}^{\tilde{n}}}(\mathcal{I})& =\frac{1}{\tilde{n}} \sum\limits_{i \in \mathbf{Good}_{n}} \P_{\widetilde{X}_i^{\tilde{n}}}(\mathcal{I}) = \frac{1}{\tilde{n}} \sum\limits_{h=1}^M \sum\limits_{i \in \mathbf{Good}_{n,t_h}}\P_{\widetilde{X}_i^{\tilde{n}}}(\mathcal{I}) \\
& \geq \frac{1}{\tilde{n}} \sum\limits_{h=1}^M |\mathbf{Good}_{n,t_h}| \P_{X^{(1-\vep)F_m}_{t_h}}(\mathcal{I}) \geq (1-\vep) \sum\limits_{h=1}^M |I_{t_h}| \P_{X^{(1-\vep)F_m}_{t_h}}(\mathcal{I})=(1-\vep)\P_{X^{(1-\vep)F_m}}(\mathcal{I}), \notag
\end{align}
where the second inequality follows from \eqref{eq:include2}.

Now we apply Lemma~\ref{lem:riordan} to $(1-\vep)F_m$ and $\widetilde{G}_{\tilde{n}}$ to conclude that

\begin{align*}
C_k\left(G_n\left(\frac{1}{n} \right) \right)& \geq C_k\left(\widetilde{G}_{\tilde{n}}\left(\frac{1}{\tilde{n}} \right) \right)  \geq (1-\vep) \tilde{n} \P_{X^{(1-2\vep)F_m}}(\mathcal{A})+o_p(n) \\
& \geq (1-2\vep) n \P_{X^{(1-2\vep)F_m}}(\mathcal{A})+o_p(n).
\end{align*}

\end{proof}

\subsection{Proof of \eqref{eq:main2}}\label{subsec:main2}
Note that if $F_m$ converges to $W$ pointwise from below, by the dominated convergence theorem it can be easily seen that 
\begin{align*}
\lim\limits_{\vep \to 0, m \to \infty} d_{\square}((1-2\vep)F_m, W)=0. 
\end{align*}
Therefore it is sufficient to show that 
\begin{align*}
\lim\limits_{n \to \infty} \P_{X^{W_n}}(\mathcal{A}) \geq \P_{X^W}(\mathcal{A}) \text{ if $\lim\limits_{n \to \infty} d_{\square}(W_n, W)=0$},
\end{align*}
which we will prove in Proposition~\ref{prop:continuity}. 

We say a branching process has property $\mathcal{B}_d$ if its root has at least $k-1$ offsprings, each of these $k-1$ offsprings has at least $k-1$ offsprings, and this occurs up to generation $d$, and let $\mathcal{B}=\lim\limits_{d \to \infty} \mathcal{B}_d$. Define functions 
\begin{align*}
\Psi_{k}(\lambda):= \P( \text{Poi}(\lambda) \geq k).
\end{align*}
For any graphon $W$, define 
\begin{align}\label{eq:defbeta}
\beta_{W} (x,d):= \P( X^W \in \mathcal{B}_d \, | \, X_0=x ), \quad \beta_W(x):= \P(X^W \in \mathcal{B} \, | \, X_0=x).
\end{align}
For $W=W_n$, we simply denote $$\beta_n(x,d):=\beta_{W_n}(x,d), \quad \beta_n(x):=\beta_{W_n}(x).$$

 \begin{lem}\label{lem:jul111:559pm} 
 Let $(W_n)_{n \in \N}$ be a sequence of graphons such that $\|W_n -W\|_\square \to 0$. Suppose that $W$ is irreducible, and $\alpha: [0,1] \rightarrow [0,1]$ is a strictly positive measurable function, i.e., $\text{Leb}(\{x: \, \alpha(x)>0\})=1.$ Fix $\vep >0$ and $\delta>0$. For large enough $n$,  there exists a subset $\mathbf{Bad}_n \subset [0,1]$ such that $\text{Leb}(\mathbf{Bad}_n) \leq \delta$, and 
 \begin{equation}\label{eqn:jul103:39pm}
(1-\vep/2) \int \alpha(y) W(x, y)\, dy \leq  \int \alpha(y) W_n(x, y)\, dy 
 \end{equation}
 for all $x \in \mathbf{Bad}_n^c$. The choice of $\mathbf{Bad}_n$ depends on $W_n,W,\vep, \delta, \alpha$.
 \end{lem}
 \begin{proof}
Take a subset 
\begin{align*}
\mathbf{Bad}_n:=\left\{x: \, (1-\vep/2)\int \alpha(y) W(x, y)\, dy > \int \alpha(y) W_n(x, y)\, dy  \right\}. 
\end{align*}
According to the definition of cut norm \eqref{eqn:jul9415pm}, we get that 
\begin{align*}
\lVert W-W_n \rVert_{\square} \geq & \int_{\mathbf{Bad}_n} \, dx\int \alpha(y) \left(W(x,y)-W_n(x,y) \right)  \, dy  \\
\geq & \frac{\vep}{2} \int_{\mathbf{Bad}_n} \, dx\int \alpha(y)  W(x,y) \, dy. 
\end{align*}
Note that $\widetilde{W}(x,y)=\alpha(x)\alpha(y)W(x,y)$ is still irreducible, and 
\begin{align*}
\lVert W-W_n \rVert_{\square} \geq \frac{\vep}{2}  \int_{\mathbf{Bad}_n} \, dx\int \alpha(y)  W(x,y) \, dy \geq  \frac{\vep}{2} \int_{\mathbf{Bad}_n} ) \, dx\int \widetilde{W}(x,y) \, dy.
\end{align*}
Therefore our claimed result directly follows from \cite[Lemma 6]{MR2599196} where the irreducibility of $\widetilde{W}$ is necessary.
\end{proof}

 \begin{lem}\label{lem:nonzero}
 Let $k \in \N$, and $W$ be an irreducible graphon  such that 
\begin{align}\label{eq:nonzerosolution}
\alpha(x) = \Psi_k\left(\int W(x,y) \alpha(y) \, dy\right)
 \end{align}
 has a non-zero solution  $\alpha(x)$, i.e., $\text{Leb}(\{x: \alpha(x)=0\})<1$. Then $\alpha$ is strictly positive, i.e., $\text{Leb}(\{x: \alpha(x)>0\})=1$.
 \end{lem}
 \begin{proof}
 Suppose $\alpha$ is a non-zero solution of \eqref{eq:nonzerosolution}. Define $\mathcal{Z}:=\{x: \, \alpha(x)=0\}$. For any $x \in \mathcal{Z}$, we have that 
 \begin{align*}
 0= \Psi_k\left(\int W(x,y) \alpha(y) \, dy\right),
\end{align*}
 which can only happen if 
 \begin{align*}
 \int_{\mathcal{Z}^c} W(x,y) \alpha(y) \, dy =0.
 \end{align*}
 Integrating the above equation over $x \in \mathcal{Z}$, we obtain that 
 \begin{align*}
 \int_{\mathcal{Z}} \, dx \int_{\mathcal{Z}^c} W(x,y) \alpha(y) \, dy =0.
 \end{align*}
 Since $\alpha>0$ over $\mathcal{Z}^c$, it violates the irreducibility of $W$. 
 \end{proof}

\begin{prop}\label{prop:continuity}
Let $(W_n)_{n \in \mathbb{N}}$ be a sequence of graphons such that $d_{\square}(W_n, W) \rightarrow 0$ with irreducible $W$. For any $\vep>0$ we have that  
\begin{equation}\label{eqn:230pm11dec20}
  \P_{X^{W_n}} (\mathcal{A}) \geq \P_{X^{(1-\epsilon)W}}(\mathcal{A})-\vep^2,
  \end{equation}
  for large enough $n$.
 If make Assumption~\ref{assume2}, then we obtain
\begin{align*}
\liminf \limits_{n \to \infty} \P_{X^{W_n}}(\mathcal{A}) \geq \P_{X^W}(\mathcal{A}).
\end{align*}
\end{prop}

\begin{proof}
 We will only prove \eqref{eqn:230pm11dec20}, since the second statement follows from this directly. Due to the equality 
 $$\P_{X^{(1-\vep) W}}(\mathcal{A})=\int \Psi_{k}\left(\int (1-\vep) W(x,y) \beta_{(1-\vep) W}(y) \, dy \right) \, dx,$$
we assume that there exists an $\vep_0>0$ such that $\text{Leb}\{ x: \, \beta_{(1-\vep_0)W}(x)>0\}>0$. Otherwise there is nothing to prove.
 Since $\beta_{(1-\vep_0)W}(x)$ is a non-zero solution of 
 \begin{align*}
     \alpha(x)=\Psi_{k-1}\left(\int (1-\vep_0) W(x,y) \alpha(y) \, dy \right), 
 \end{align*}
$\beta_{(1-\vep_0)W}(x)$ is strictly positive according to Lemma~\ref{lem:nonzero}. Fix  any $\vep \in (0,\vep_0)$. We first prove the following statement by induction:  for each $d \geq 1$ and $\delta >0$, there exist subsets $\mathbf{Bad}_{n,d} \subset [0,1]$ for large enough $n$ such that $\text{Leb}(\mathbf{Bad}_{n,d})< \delta$ and 
\begin{align}\label{eq:beta}
\beta_n(x,d)\geq \beta_{(1-\vep)W}(x,d), \quad \text{for any $x \in \mathbf{Bad}^c_{n,d}$, $d \geq 1$}.
\end{align}

 Applying Lemma~\ref{lem:jul111:559pm} with $\alpha(y)=1, \, \forall y \in [0,1]$, we obtain some $\mathbf{Bad}_{n,1}$ with $\text{Leb}(\mathbf{Bad}_{n,1}) \leq \delta $ such that $x \in \mathbf{Bad}_{n,1}^c$ implies $\int W_n(x,y) \, dy \geq (1-\vep/2) \int W(x,y) \, dy$. It follows that
 $$\beta_n(x,1) =\Psi_{k-1}\left(\int W_n(x,y) \, dy\right) \geq \Psi_{k-1}\left((1-\vep/2)\int W(x,y) \, dy\right) $$
 
 $$
 \geq \Psi_{k-1}\left((1-\vep)\int W(x,y) \, dy\right) = \beta_{(1-\vep)W}(x,1).$$
Assume our assertion holds for $d-1$ and $\delta'>0$, where $\delta'$ is to be chosen later. Let us now prove our claim for $d$ and $\delta>0$. Noting that 
\begin{align*}
\beta_n(x,d) &= \Psi_{k-1}\left(\int W_n(x,y) \beta_n(y,d-1) \, dy\right),   \\
\beta_{(1-\vep)W}(x,d)& =\Psi_{k-1}\left(\int (1-\vep) W(x,y) \beta_{(1-\vep)W}(y,d-1) \, dy\right), 
\end{align*}
it is enough to show there exists $\mathbf{Bad}_{n,d}$ such that $x \in\mathbf{Bad}_{n,d}^c$ implies  
\begin{align}\label{eq:integralinduction}
\int W_n(x,y) \beta_n(y,d-1) \, dy \geq \int (1-\vep) W(x,y) \beta_{(1-\vep)W}(y,d-1) \, dy.
\end{align}
Applying Lemma \ref{lem:jul111:559pm} with $\alpha(y)=\beta_{(1-\vep)W}(y,d-1)>0$, we obtain for large enough $n$ a subset $\widetilde{\mathbf{Bad}}_{n,d}$  with $\text{Leb}(\widetilde{\mathbf{Bad}}_{n,d}) \leq \delta/2$ such that $x \in \widetilde{\mathbf{Bad}}^c_{n,d}$ implies that 
\begin{align*}
    \int W_n(x,y)\beta_{(1-{\vep})W}(y,d-1) \, dy \geq  (1-\vep/2) \int W(x,y)\beta_{(1-\vep)W}(y,d-1) \, dy.
\end{align*}
Since $\beta_{(1-\vep)W}(x)$ is strictly positive, 
we get that
\begin{align*}
\lim\limits_{\xi \searrow 0} \text{Leb}\left(\left\{x: \,\beta_{(1-\vep)W}(x) \geq \xi  \right\} \right)=1. 
\end{align*}
Therefore, there exists a subset $\mathbf{Bad} \subset [0,1]$ and a positive constant $\xi$ such that $\text{Leb} (\mathbf{Bad}) \leq \delta/4$ and for any $x \in \mathbf{Bad}^c$
\begin{align*}
 \int W(x,y) \beta_{(1-\vep)W}(y, d-1)\,dy \geq \beta_{(1-\vep)W}(x) \geq \xi.
\end{align*}
By induction, it follows that for $x \in \mathbf{Bad}_{n,d}^c:=\mathbf{Bad}_{n,d-1}^c \cap \widetilde{\mathbf{Bad}}_{n,d}^c \cap \mathbf{Bad}^c$
\begin{align*}
\int W_n(x,y) \beta_n(y,d-1) \, dy  & \geq \int_{y \in \mathbf{Bad}_{n,d-1}^c} W_n(x,y)\beta_{(1-{\vep})W}(y,d-1) \, dy \\
& \geq \int W_n(x,y)\beta_{(1-\vep)W}(y,d-1) \, dy - \delta' \bar{a} \\
&  \geq (1-\vep/2) \int W(x,y)\beta_{(1-\vep)W}(y,d-1) \, dy - \delta' \bar{a},
\end{align*} 
where the first two inequalities are due to our induction hypothesis and the third inequality follows from our choice of $\widetilde{\mathbf{Bad}}_{n,d}$. Choosing a small enough $\delta' \in (0,\delta/4)$ such that $\delta ' \bar{a} \leq \frac{\vep \xi}{2}$, it is clear that  \eqref{eq:integralinduction} holds for $x \in \mathbf{Bad}_{n,d}^c$ with $\text{Leb}(\mathbf{Bad}_{n,d}) \leq \delta$.

Now we prove that $ \P_{X^{W_n}} (\mathcal{A}) \geq \P_{X^{(1-\epsilon)W}}(\mathcal{A})-\vep^2$. Note that
$$
\P_{X^{W_n}}(\mathcal{A}_d) = \int \Psi_k\left(\int W_n(x,y) \beta_n(y,d-1) \, dy \right)\, dx,
$$
and
$$
\P_{X^{(1-\vep)W}}(\mathcal{A}_d) = \int \Psi_k\left(\int (1-\vep)W(x,y) \beta_{(1-\vep)W}(y,d-1) \, dy \right) \, dx.
$$
Choosing $\delta =\vep^2$, there exists for large enough $n$ a subset $\mathbf{Bad}_{n,d}$ with $\text{Leb}(\mathbf{Bad}_{n,d})\leq \delta$ such that $x \in \mathbf{Bad}_{n,d}^c$ implies
\begin{align*}
    \Psi_k\left(\int W_n(x,y) \beta_n(y,d-1) \, dy \right) \geq \Psi_k\left(\int (1-\vep)W(x,y) \beta_{(1-\vep)W}(y,d-1) \, dy \right).
\end{align*}
Due to $\Psi_k(x) \leq 1$ for all $x \geq 0$, it can be easily verified that
\begin{align*}
  \P_{X^{W_n}}(\mathcal{A}_d) \geq \P_{X^{(1-\vep)W}}(\mathcal{A}_d)-\vep^2. 
\end{align*}
Letting $d \to \infty$ in the above inequality, we conclude the result.
\end{proof}

\subsection*{Acknowledgment}We thank Remco van der Hofstad for his comments that connects our work to the emergence of a giant $k$-core, which lead to Remark \ref{rmk:351pm16dec20}. 
E. B. is partially supported by the National Science Foundation under grant DMS-2106556 and by the
Susan M. Smith chair.
S.C. is partially supported by the Netherlands Organisation for Scientific Research (NWO) through Gravitation-grant NETWORKS-024.002.003.

\bibliographystyle{siam}
\bibliography{kcore_densegraph.bib}
\end{document}